\documentclass[12pt,oneside,reqno]{amsart}

\usepackage{tikz}
\usetikzlibrary{arrows.meta}
\usetikzlibrary{bending}

\usepackage{amssymb}

\usepackage{hyperref}
\usepackage{amsxtra}
\usepackage{amsopn}
\usepackage{amsmath,amsthm,amssymb}
\usepackage{color}
\usepackage{amscd}
\usepackage{pifont}
\usepackage{amsfonts}
\usepackage{latexsym}
\usepackage{verbatim}
\usepackage{pb-diagram}
\usepackage{tikz}

\newcommand{\fr}{\mathfrak}

\newcommand{\op}{\operatorname}

 \newtheorem{lemma} {Lemma} [section]
\newtheorem{theorem}[lemma]{Theorem} 
\newtheorem{remark}[lemma] {Remark} 
\newtheorem{proposition} [lemma]{Proposition}  
\newtheorem{definition}[lemma] {Definition} 
\newtheorem{corollary}[lemma] {Corollary} 
\newtheorem{example}[lemma] {Example}

\begin{document}
\title[A characterization of perfect Leibniz algebras]{A characterization of perfect Leibniz algebras} 

\author{Nikolaos Panagiotis Souris}
\address{University of Patras, Department of Mathematics, University Campus, 26504, Rio Patras, Greece}
\email{nsouris@upatras.gr}


\maketitle

\begin{abstract}
Leibniz algebras are non-antisymmetric generalizations of Lie algebras that have attracted substantial interest due to their close relation with the latter class. A Leibniz algebra $A$ is called perfect if it coincides with its derived subalgebra $A^2$. As a generalization of an analogous result for Lie algebras, we show that perfect Leibniz algebras, of arbitrary dimension and over any field, are characterized among Leibniz algebras by the property that they are ideals whenever they are embedded as subideals. Equivalently, we prove that perfect Leibniz algebras are precisely those Leibniz algebras such that whenever they are embedded as ideals, they are characteristic ideals, i.e., they are invariant under all derivations of the ambient algebra.  We apply the above to prove certain inclusion relations for derivation algebras of perfect Leibniz algebras.

\medskip 



\end{abstract}

\section{Introduction}

Leibniz algebras were introduced by Bloh in 1965 (\cite{Bloh}), but made more widely known in the mid 90's through the works of Loday and Pirasvili (\cite{Lod}, \cite{LodPir}).  A \emph{left Leibniz algebra} over a field $\mathbb F$ is a vector space $A$ over $\mathbb F$, endowed with a bilinear product $[ \ , \ ]:A\times A\rightarrow A$ such that for any $x\in A$, the left multiplication $L_x:=[x, \ _- \ ]$ is a derivation, that is

\begin{equation}\label{Jacobi}\big[x,[y,z]\big]=\big[[x,y],z\big]+\big[y,[x,z]\big],  \ \ \makebox{for all} \ \ x,y,z\in A.\end{equation}

\noindent Similarly, \emph{right Leibniz algebras} are defined by requiring that right multiplications are derivations. The \emph{Leibniz kernel} of $A$ is the characteristic (\cite{BoMiSti}) ideal $Leib(A)=\{[x,x]:x\in A\}$, and coincides with the smallest ideal $I$ of $A$ such that $A/I$ is a Lie algebra.  Thus Lie algebras are precisely those Leibniz algebras $A$ with $Leib(A)=\{0\}$.  This simple connection between Lie and Leibniz algebras has stimulated several studies, that analyze the extent to which the rich theory of the former class can be generalized to the latter. For example, both concepts of \emph{nilpotent} and \emph{solvable} Lie algebras extend naturally to Leibniz algebras.  A finite-dimensional Leibniz algebra $A$ is called \emph{semisimple} if the radical $rad(A)$ (the maximal solvable ideal of $A$) coincides with $Leib(A)$ (\cite{GoVi}). Key structural properties, such as Engel's theorem (\cite{Aya}) and the Levi - Malcev theorem in characteristic zero (\cite{Barn}), are also true for Leibniz algebras, yet obstructions may appear due to the lack of anti-symmetry (e.g. the Levi factors are not unique up to conjugation). Much like their antisymmetric counterparts, Leibniz algebras have interactions with mathematical physics (e.g. in gauge theories \cite{BoHo}), while they were recently interpreted in characteristic zero as differential graded Lie algebras (\cite{Most}).

In this paper, we consider the class of \emph{perfect Leibniz algebras}, namely those Leibniz algebras $A$ that coincide with their derived subalgebra $A^2:=[A,A]$.  Perfect Lie algebras form an interesting class, studied among other reasons for their cohomology (\cite{BuWa}, \cite{Pi}) and their interactions with mathematical physics (\cite{Sal}, \cite{Stu}).  However, the class of perfect Leibniz algebras remains rather unexplored. Any semisimple Leibniz algebra in characteristic zero is perfect, but there exist perfect Leibniz algebras which are neither semisimple nor Lie algebras (see Example \ref{secexample}). 

Schenkman, in his Ph.D. thesis (see \cite{Sh}, Corollary 1), proved that if a perfect Lie algebra $\fr{g}$ is a \emph{subideal} of a Lie algebra $\fr{k}$ (i.e., there exist finitely many subalgebras $\fr{k}_0,\dots ,\fr{k}_n$ such that $\fr{g}=\fr{k}_0\unlhd \fr{k}_1 \unlhd\cdots \unlhd \fr{l}_n=\fr{k}$) then $\fr{g}$ is an ideal of $\fr{k}$ (i.e., $\fr{g}\unlhd \fr{k}$).  It then follows that any perfect ideal is a characteristic ideal.  The author proved a converse of this property in \cite{So}, namely if a Lie algebra $\fr{g}$ is an ideal whenever it is embedded as a subideal then $\fr{g}$ is perfect. Schenkman's result is generalized for perfect Leibniz algebras by Misra, Stitzinger, and Yu in \cite{MiSiYu}, Corollary 2.4.  In this paper, we obtain the converse for perfect Leibniz algebras, essentially characterizing perfect Leibniz algebras by this property.  In particular, denoting by $\unlhd$ the relation of being an (two - sided) ideal in Leibniz algebras, we prove the following.

\begin{theorem}\label{th1} Let $A$ be a Leibniz algebra over a field $\mathbb F$. The following are equivalent:\\

\noindent \emph{(i)} The algebra $A$ is perfect.\\
\noindent \emph{(ii)} For all Leibniz algebras $K$ and $M$ such that $A\unlhd K\unlhd M$, it follows that $A\unlhd M$.\\
\noindent \emph{(iii)} For all Leibniz algebras $K$ with $A\unlhd K$, it follows that $A$ is a characteristic ideal of $K$.  \end{theorem}

Although we established in \cite{So} a similar characterization for perfect Lie algebras based on the notion of the \emph{holomorph} (the semidirect product of a Lie algebra with its derivation algebra), the proof for Leibniz algebras cannot be obtained by a straightforward generalization.  The main obstruction lies in the fact that there does not exist a notion of a holomorph of Leibniz algebras $A$ that includes $A$ as an ideal (see short discussion in Section \ref{secdeflie}). We bypass this obstruction by introducing the \emph{Lie - holomorph} of a Leibniz algebra (Section \ref{secdeflie}), constructed by the ideal of those derivations that satisfy anti-symmetry conditions.  Theorem \ref{th1} and some simple corollaries are proven in Section \ref{mainsection}. 

We recall that a Lie algebra is called \emph{complete} if is has trivial center and only inner derivations. In \cite{SuZhu}, Su and Zhu proved that the Lie algebra of derivations $D(\fr{g})$ of any perfect Lie algebra $\fr{g}$ with trivial center is complete.  The last fact is equivalent to $D(D(\fr{g}))=D(\fr{g})$, or $D(D(\fr{g}))\subseteq D(\fr{g})$, given that the reverse inclusion always holds. Motivated by this property, in Section \ref{seccomplete}, we apply Theorem \ref{th1} aiming to obtain analogous inclusion relations for derivation algebras of perfect Leibniz algebras (among other results of possible independent interest). More specifically, let $D(A)$ be the (Lie) algebra of derivations of a Leibniz algebra $A$ and let $I$ be the ideal (c.f. Proposition \ref{ideal})
  
  \begin{equation*}\label{I} I:=\{f\in D(A): Im(f)\subseteq Leib(A)\},\end{equation*}
  
  \noindent defined also by Boyle, Misra and Stitzinger (see Theorem 4.3. in \cite{BoMiSti}).  For $n\in \mathbb N$, denote by $D^n(D(A)/I)$ the Lie algebras defined inductively by $D^0(D(A)/I):=D(A)/I$ and $D^{n+1}(D(A)/I):=D\big(D^n(D(A)/I)\big)$. We prove the following.

\begin{theorem}\label{complete} Let $A$ be a perfect Leibniz algebra such that the center of $A/Leib(A)$ is trivial.  Then for all $n\in \mathbb N$, the derivation algebra $D^n(D(A)/I)$ can be embedded as a subalgebra of the derivation algebra $D(A/Leib(A))$.\end{theorem}

We remark that the result of Su and Zhu can be obtained as an immediate corollary of Theorem \ref{complete} (see Section \ref{seccomplete}). Throughout this paper, we will mostly follow the notation in \cite{BoMiSti}. Unless otherwise stated, all Leibniz ideals are assumed to be two-sided, and we consider only left Leibniz algebras without loss of generality.

\section{The Lie - holomorph of a Leibniz algebra}\label{secdeflie}

 The \emph{holomorph of a Lie algebra} $\frak{g}$ is defined as the semidirect product $Hol(\fr{g}):=\frak{g}\rtimes D(\frak{g})$, equipped with the Lie bracket $\big[(x,f),(y,g)\big]:=\big([x,y]+f(y)-g(x),[f,g]\big)$, for $x,y\in \fr{g}$ and $f,g\in D(\fr{g})$.  An important aspect that we took advantage of in \cite{So} is that $\fr{g}$ can be embedded as an ideal into its holomorph.  Unfortunately, this holomorph construction does not extend directly to Leibniz algebras, as it may not yield a Leibniz algebra.  Instead, the authors in \cite{BoMiSti} define the \emph{holomorph of a Leibniz algebra} $A$ to be the vector space $hol(A):=A\oplus D(A)$, equipped with the product

\begin{equation*} [x+f,y+g]:=[x,y]+f(y)+[L_x,g]+[f,g],\end{equation*}

\noindent for $x,y\in A$ and $f,g\in D(A)$.  This new notion of a holomorph allows the generalization of several results from complete Lie algebras to complete Leibniz algebras (\cite{BoMiSti}).  However, a downside of the above construction for our purpose is that $A$ cannot be embedded as an ideal into $hol(A)$.  To bypass this obstruction, we consider the \emph{Lie - center} of $A$ (we adopt the terminology and notation in \cite{Cas}), namely the ideal

\[ Z_{Lie}(A):=\{x\in A:[x,y]+[y,x]=0 \ \ \makebox{for all} \ \ y\in A\}.  \]

\noindent Accordingly, we define the space of \emph{Lie - derivations} of $A$ by

\begin{eqnarray*} D_{Lie}(A)&:=&\{f\in D(A):Im(f)\subseteq Z_{Lie}(A)\}.\end{eqnarray*}

\noindent It follows immediately that if $A$ is a Lie algebra, then $D_{Lie}(A)$ coincides with $D(A)$. Moreover, we have the following result.

\begin{proposition}\label{liecentr}Let $A$ be a Leibniz algebra over a field $\mathbb F$. Then:\\

\noindent \emph{(i)} The Lie - center $Z_{Lie}(A)$ is a characteristic ideal of $A$.\\
\noindent \emph{(ii)} The space $D_{Lie}(A)$ of Lie - derivations is an ideal of $D(A)$.\end{proposition}

\begin{proof}  For (i), the fact that $Z_{Lie}(A)$ is an ideal of $A$ has been proven in \cite{Cas}.  To show that it is a characteristic ideal, let $f\in D(A)$ and $x\in Z_{Lie}(A)$.  Since $f$ is a derivation and using the definition of $Z_{Lie}(A)$, for all $y\in A$ we have

\begin{equation*}[f(x),y]+[y,f(x)]=f\big([x,y]+[y,x]\big)-\big([x,f(y)]+[f(y),x]\big)=0,\end{equation*}

\noindent  Therefore, $f(x)\in Z_{Lie}(A)$.  For part (ii), let $f\in D_{Lie}(A)$, $g\in D(A)$ and $x,y\in A$. We have 

\begin{eqnarray*} \big[[f,g](x),y\big]+\big[y,[f,g](x)\big]&=&\big([f(g(x)),y]+[y,f(g(x))]\big)-[g(f(x)),y]-[y,g(f(x))]\\
&=&-[g(f(x)),y]-[y,g(f(x))]\\
&=&-g\big([f(x),y]+[y,f(x)]\big)+\big([f(x),g(y)]+[g(y),f(x)]\big)=0, \end{eqnarray*}

\noindent where the second equality holds because $f(g(x))\in Z_{Lie}(A)$, the third equality holds because $g$ is a derivation, and the final equality holds because $f(x)\in Z_{Lie}(A)$.  We conclude that $[f,g]\in D_{Lie}(A)$ and thus $D_{Lie}(A)$ is an ideal of $D(A)$.\end{proof}

 In turn, we use the Lie - derivations to define the Lie - holomorph of $A$.  

\begin{definition}\label{deflie}Let $A$ be a Leibniz algebra over a field $\mathbb F$. The Lie - holomorph of $A$ is the semidirect product $hol_{Lie}(A):=A\rtimes D_{Lie}(A)$, equipped with the bilinear product $\big[(x,f),(y,g)]:=\big([x,y]+f(y)-g(x),[f,g]\big)$ for $x,y\in A$ and $f,g\in D_{Lie}(A)$.\end{definition}

It turns out that the Lie - holomorph of $A$ is a Leibniz algebra that contains $A$ as an ideal.  

\begin{proposition}\label{prop1}Let $A$ be a Leibniz algebra over a field $\mathbb F$.  Then the following are true:\\

\noindent \emph{(i)} The Lie - holomorph $hol_{Lie}(A)$ is a Leibniz algebra. \\
\noindent \emph{(ii)} The algebra $A$ can be embedded as an ideal into $hol_{Lie}(A)$ by $A\rightarrow A\times \{0\}$.\end{proposition}

The proof of Proposition \ref{prop1} can be verified by straightforward computations and is omitted. Note that if $A$ is a Lie algebra then the Lie - holomorph of $A$ coincides with its holomorph as a Lie algebra, i.e., $hol_{Lie}(A)=Hol(A)$.

\section{A characterization of perfect Leibniz algebras}\label{mainsection}

We begin this section with a general construction of perfect Leibniz algebras that are neither semisimple nor Lie algebras.

\begin{example}\label{secexample}
We consider the following construction (see for example \cite{KiWe}): Let $(\fr{g}, [ \ , \ ]_{\fr{g}})$ be a Lie algebra over a field of characteristic zero, and let $V$ be a $\fr{g}$ - module with a left action $\fr{g}\times V\rightarrow V$, $(x,v)\mapsto x\cdot v$. Define a binary product $[ \ , \ ]$ on $\fr{g}\times V$ by $\big[(x,u),(y,v)\big]:=\big([x,y]_{\fr{g}},x\cdot v\big)$.  Then $\big(\fr{g}\times V, [ \ , \ ]\big)$ is a Leibniz algebra, called the hemisemidirect product of $\fr{g}$ and $V$, and denoted by $\fr{g}\ltimes_HV$.  If $\fr{g}$ acts non - trivially on $V$ then $\fr{g}\ltimes_HV$ is not a Lie algebra (see e.g. Lemma 1.1 (a) in \cite{Fel}). Moreover, by Lemma 1.1 (b) in \cite{Fel}, the product $\fr{g}\ltimes_HV$ is a perfect Leibniz algebra if and only if $\fr{g}$ is a perfect Lie algebra and $\fr{g}\cdot V=V$.

Now assume that $\fr{g}$ is a finite-dimensional non-semisimple perfect Lie algebra (see e.g. the semi-direct products in \cite{BuWa}). Then $\fr{g}$ admits a finite-dimensional, non-trivial and irreducible module $V$  (To see this, consider the Levi-Malcev decomposition $\fr{g}=\fr{s}+\fr{r}$, where $\fr{s}$ is semisimple and $\fr{r}$ is the radical of $\fr{g}$, and observe that since $\fr{g}$ is perfect then $\fr{s}$ is non-zero and $\fr{r}$ is a non-trivial $\fr{s}$-module. Then choose an irreducible non-trivial $\fr{s}$-module $V\subseteq \fr{r}$ and observe that $V$ can be made into an irreducible non-trivial $\fr{g}$-module by defining the action of $\fr{r}$ on $V$ to be trivial).  Therefore $\fr{g}\ltimes_HV$ is a perfect Leibniz algebra that is not a Lie algebra. Finally, since $\fr{g}$ is not semisimple, Theorem 2.1. in \cite{Fel} implies that $\fr{g}\ltimes_HV$ is not a semisimple Leibniz algebra.

\end{example}

We proceed with the main aim of this section, which is to prove Theorem \ref{th1}.  We also give some simple corollaries. Our main argument is based on the following criterion that implements the Lie - holomorph.

\begin{proposition}\label{mainprop}Let $A$ be a Leibniz algebra over a field $\mathbb F$.  Assume that for all Leibniz algebras $K$ such that $A\unlhd K$, and for all $f\in D_{Lie}(K)$, we have $f(A)\subseteq A$.  Then $A$ is perfect.\end{proposition}

\begin{proof} Assume on the contrary that $A$ is not perfect.  Then the algebra $A/[A,A]$ is abelian and not trivial.  We will construct a suitable Leibniz algebra $K$ containing $A$ as an ideal, and such that there exists a derivation $f\in D_{Lie}(K)$ that does not leave $A$ invariant.  To this end, let $\pi:A\rightarrow A/[A,A]$ be the canonical projection and consider the direct product $K:=A\times (A/[A,A])$, equipped with the bilinear product 

\begin{equation*} \big[(x,\pi(y)), (z,\pi(w))]:=\big([x,z],[\pi(y),\pi(w)]\big)=\big([x,z],0)\big), \end{equation*}
\noindent  for $x,z\in A$ and $\pi(y),\pi(w)\in A/[A,A]$, $y,w\in A$.  It is clear that $(K, [ \ , \ ])$ is a Leibniz algebra and $A=A\times \{0\}$ is an ideal of $K$. Define a linear endomorphism $f:K\rightarrow K$ by 

\[ f\big((x,\pi(y))\big):=(0,\pi(x))\in \{0\}\times (A/[A,A]),  \ \ x,y\in A.\]

\noindent Then for all $x,y,z,w\in A$ we have on the one hand that 

\begin{equation}\label{eq1} f\big(\big[(x,\pi(y)), (z,\pi(w))\big]\big)=f\big([x,z],0\big)=\big(0,\pi([x,z])\big)\in \{0\}\times \pi([A,A])=\{0\}. \end{equation}

\noindent On the other hand,  using the definition of $f$, it can be easily verified that

\begin{equation}\label{eq2} \big[f((x,\pi(y))),(z,\pi(w))\big]+\big[(x,\pi(y)),f((z,\pi(w)))\big]=0.\end{equation}

\noindent Equations \eqref{eq1} and \eqref{eq2} imply that $f\in D(K)$.  Moreover, 

\begin{equation*}\big[f((x,\pi(y))),(z,\pi(w))\big]+\big[(z,\pi(w)),f((x,\pi(y)))\big]=0,\end{equation*}

\noindent for all $x,y,z,w\in A$, and thus $f\in D_{Lie}(K)$. Since $A/[A,A]$ is non-trivial, we can choose a $x\in A$ such that $\pi(x)\neq 0$, and thus $0\neq f((x,0))=(0,\pi(x))\in \{0\}\times (A/[A,A])$.  Therefore, $f(A\times \{0\})\not\subseteq A\times \{0\}$, which is a contradiction.  We conclude that $A/[A,A]$ is trivial, i.e., $A$ is perfect.\end{proof} 

We proceed to prove Theorem \ref{th1}.\\

\noindent \emph{Proof of Theorem \ref{th1}.}  We will prove that statement (i) is equivalent to both statements (ii) and (iii). Firstly, we prove that (i) implies (ii). This can be deduced from the results in \cite{MiSiYu}, but we prove it here for completeness using simply the identity \eqref{Jacobi}.  Assume that $A$ is perfect and $A\unlhd K\unlhd M$.  For all $x_i,y_i\in A$ and $z\in M$, we have 

\begin{equation*}\big[[x_i,y_i],z\big]=\big[x_i,[y_i,z]\big]-\big[y_i,[x_i,z]\big]\in \big[A,[A,M]\big]\subseteq [A,K]\subseteq A.\end{equation*}
 
\noindent Similarly,

\begin{equation*}\big[z,[x_i,y_i]\big]=\big[[z,x_i],y_i\big]+\big[x_i,[z,y_i]\big]\in \big[[M,A],A\big]+\big[A,[M,A]\big]\subseteq [K,A]+[A,K]\subseteq A.\end{equation*}

\noindent Since any element of $A=[A,A]$ can be written as a finite sum of elements of the form $[x_i,y_i]$, $x_i,y_i\in A$, the above equations imply that $A$ is both a left and right ideal of $M$, and hence an ideal of $M$. This proves that (i) implies (ii).  To prove that (ii) implies (i), we will use Proposition \ref{mainprop}. Let $K$ be a Leibniz algebra such that $A\unlhd K$. By part (ii) of Proposition \ref{prop1}, we have a subideal sequence  $A\times \{0\}\unlhd K\unlhd hol_{Lie}(K)$. Let $x\in A$ and $f\in D_{Lie}(K)$.  The product in $hol_{Lie}(K)$ from Definition \ref{deflie} yields $\big[(x,0),(0,f)\big]=(-f(x),0)$. On the other hand, our assumption and the subideal sequence  $A\times \{0\}\unlhd K\unlhd hol_{Lie}(K)$ imply that $A\times \{0\}$ is an ideal of $hol_{Lie}(K)$, and thus $(-f(x),0)\in A\times \{0\}$, i.e., $f(x)\in A$ for all $x\in A$ and $f\in D_{Lie}(K)$.  Since the algebra $K$ is arbitrary, Proposition \ref{mainprop} implies that $A$ is perfect.

Finally, we prove the equivalence of (i) and (iii). Assuming that (i) holds, suppose that $A\unlhd K$ for some Leibniz algebra $K$.  Then for all $x_i,y_i\in A$ and $f\in D(K)$, we have $f\big([x_i,y_i]\big)=[f(x_i),y_i]+[x_i,f(y_i)]\in [A,A]=A$.  Since any element of $A$ is written as a finite sum of elements of the form $[x_i,y_i]$, $x_i,y_i\in A$, the last equation implies that $f(A)\subseteq A$, and hence $A$ is a characteristic ideal of $K$.  Conversely, if (iii) holds then for all Leibniz algebras $K$ such that $A\unlhd K$ and for all $f\in D_{Lie}(K)\subseteq D(K)$, it follows that $f(A)\subseteq A$.  Proposition \ref{mainprop} then implies that $A$ is perfect.\qed \\

\noindent  By a simple inductive argument, Theorem \ref{th1} implies the following.

\begin{corollary}\label{corol1} Let $A$ be a Leibniz algebra over a field $\mathbb F$. The following are equivalent:\\

\noindent \emph{(i)} The algebra $A$ is perfect.\\
\noindent \emph{(ii)} For all Leibniz algebras $K$, the algebra $A$ is a subideal of $K$ if and only if it is an ideal of $K$ (i.e., $A\unlhd K$).\end{corollary}

\noindent To conclude this section, recall that if $A$ is a Leibniz subalgebra of a Leibniz algebra $K$, the \emph{normalizer} $N_K(A)$ is the largest Leibniz subalgebra of $K$ containing $A$ as an ideal.  An immediate consequence of Theorem \ref{th1} is that the normalizer tower $A\unlhd N_K(A)\unlhd N_{K}(N_K(A))\unlhd \cdots $, of any perfect Leibniz subalgebra of a Leibniz algebra $K$, terminates at the second step, i.e., the algebra $N_K(A)$ is self - normalizing.

\begin{corollary}\label{norm}Let $A$ be a perfect Leibniz subalgebra of a Leibniz algebra $K$.  Then the normalizer $N_K(A)$ of $A$ in $K$ is self - normalizing, that is $N_K(N_K(A))=N_K(A)$.\end{corollary}

\begin{proof} In view of Theorem \ref{th1}, the subideal sequence $A\unlhd N_K(A)\unlhd N_{K}(N_K(A))$ implies that $A\unlhd N_{K}(N_K(A))$. Since $N_K(A)$ is the largest subalgebra of $K$ containing $A$ as an ideal, it follows that $N_K(N_K(A))=N_K(A)$.\end{proof}

\section{An application to derivation algebras}\label{seccomplete}
 The main aim of this section is to prove Theorem \ref{complete}. Firstly, we state some auxiliary results.  For a Leibniz subalgebra $A$ of a Leibniz algebra $K$, the \emph{left - centralizer} of $A$ in $K$ is the Leibniz subalgebra of $K$, given by 
 
 \[ C^l_{K}(A)=\{x\in K:[x,y]=0 \ \ \makebox{for all} \ \  y\in A\}. \]
 
 \noindent We note that if $A,K$ are Lie algebras, then the left centralizer $C^l_{K}(A)$ is just the centralizer $C_{K}(A)=\{x\in K:[x,y]=[y,x]=0 \ \ \makebox{for all} \ \  y\in A\}$.  Theorem \ref{th1}, along with the following proposition will be our main tools.

\begin{proposition}\label{lemma10} Assume that a Leibniz algebra $A$ is a characteristic ideal of a Leibniz algebra $K$, and $C^l_{K}(A)=\{0\}$.  Then the Lie algebra $D(K)$ can be embedded as a subalgebra of $D(A)$.\\
 \end{proposition}

\begin{proof} Define a map $\phi:D(K)\rightarrow D(A)$ with $\phi(f):=\left.f\right|_{A}$, $f\in D(K)$.  Since $A$ is a characteristic ideal of $K$, the map $\phi$ is well - defined. Also $\phi$ is a homomorphism of Lie algebras.  Indeed, for all $f,g\in D(K)$ and $x\in A$, we have $f(x), g(x)\in A$, and hence $\phi([f,g])(x)=\left.f\right|_{A}(\left.g\right|_{A}(x))-\left.g\right|_{A}(\left.f\right|_{A}(x))=[\left.f\right|_{A},\left.g\right|_{A}](x)$, i.e., $\phi([f,g])=[\phi(f),\phi(g)]$.  Now let $f\in Ker(\phi)$. Since $A$ is a characteristic ideal of $K$, for all $x\in K$ and $y\in A$, we have

\begin{eqnarray*}0&=&\phi(f)([x,y])=\left.f\right|_{A}([x,y])=f([x,y])=[f(x),y]+[x,f(y)]\\
&=&[f(x),y]+[x,\left.f\right|_{A}(y)]=[f(x),y]+[x,\phi(f)(y)]=[f(x),y].\end{eqnarray*}

\noindent The above equation yields $f(x)\in C^l_{K}(A)$ for all $x\in K$. Since $C^l_{K}(A)=\{0\}$, we have $Ker(\phi)=\{0\}$, i.e., $\phi$ is injective. \end{proof}

Before we state the next result, let us recall some known facts for a Leibniz algebra $A$ (see for Section 3 in (\cite{BoMiSti})). Let $L(A)=\{L_x:x\in A\}$ be the ideal of left multiplications in $A$. Then

\begin{equation}\label{Lx}[f,L_x]=L_{f(x)}\ \ \makebox{for all} \ \ f\in D(A), \ L_x\in L(A).\end{equation}

\noindent The left center of $A$ is the characteristic ideal $Z^l(A)=\{x\in A: [x,y]=0 \ \ \makebox{for all} \ \ y\in A\}$. Moreover, $Leib(A)\subseteq Z^l(A)$ and $L(A)$ is isomorphic to $A/Z^l(A)$. Finally, if the center $Z(A/Leib(A))$ of $A/Leib(A)$ is trivial then the ideals $Z^l(A)$ and $Leib(A)$ coincide, and thus $L(A)$ is isomorphic to $A/Leib(A)$.

\begin{proposition}\label{ideal} The subspace $I:=\{f\in D(A): Im(f)\subseteq Leib(A)\}$ of $D(A)$ is an ideal of $D(A)$. Moreover, if the center of $A/Leib(A)$ is trivial then $I$ is a characteristic ideal.\end{proposition}

\begin{proof} The space $I$ is clearly a subalgebra of $D(A)$. To see that it is an ideal, let $f\in I$, $g\in D(A)$ and $x\in A$, so that $[f,g](x)=f(g(x))-g(f(x))$. By the definition of $I$, we have $f(g(x))\in Leib(A)$.  Since $Leib(A)$ is a characteristic ideal of $A$, it follows that $g(f(x))\in g(Leib(A))\subseteq Leib(A)$.  Therefore, $[f,g]\in I$, i.e., $I$ is an ideal of $D(A)$.  Now assume that the center of $A/Leib(A)$ is trivial.  To prove that $I$ is a characteristic ideal of $D(A)$, let $F\in D^2(A)=D(D(A))$ and $f\in I$. We will show that $F(f)\in I$.  To this end, let $x\in A$ and note by relation \eqref{Lx} that $[f,L_x]=L_{f(x)}=0$, given that $f(x)\in Leib(A)=Z^l(A)$.  Therefore, $F([f,L_x])=0$, i.e., $[F(f),L_x]+[f,F(L_x)]=0$. Since $I$ is an ideal of $D(A)$, the second term of the left - hand side of the last equation lies in $I$, while the first term is equal to $L_{F(f)(x)}$.  We deduce that $L_{F(f)(x)}\in I$ for all $x\in A$, i.e $[F(f)(x),y]\in Leib(A)$ for all $x,y\in A$.  But the last relation implies that $F(f)(x)+Leib(A)$ lies in the center of $A/Leib(A)$, which is trivial. We conclude that $F(f)(x)\in Leib(A)$ for all $x\in A$, that is $F(f)\in I$.   \end{proof}

 We will also use the following.

\begin{lemma}\label{lemma11} Let $A$ be a Leibniz algebra such that the center of $A/Leib(A)$ is trivial. Let $I$ be the ideal of $D(A)$ given in Proposition \ref{ideal}.  Then the center of $D(A)/I$ is trivial. \end{lemma}
 
 \begin{proof} Let $f+I\in Z(D(A)/I)$. Using relation \eqref{Lx}, for all $x\in A$ we have $0=[f+I,L_x+I]=L_{f(x)}+I$. The last equation yields $L_{f(x)}\in I$, i.e., $[f(x),y]\in Leib(A)$ for all $y\in A$.  Therefore, $f(x)+Leib(A)\in Z(A/Leib(A))=\{0\}$, and thus $Im(f)\subseteq Leib(A)$, i.e., $f+I=0$. \end{proof}

\begin{lemma}\label{lemma12} Let $A$ be a Leibniz algebra with $Z(A/Leib(A))=\{0\}$. Then the Lie algebra $L(A)$ can be embedded as an ideal of $D(A)/I$.  Moreover, the centralizer $C_{D(A)/I}(L(A))$ is trivial.\end{lemma}

\begin{proof} We consider the map $\phi:L(A)\rightarrow D(A)/I$ with $L_x\mapsto L_x+I$, $x\in A$.  Observe that $[L_x+I,L_y+I]=L_{[x,y]}+I$, and hence $\phi$ is a homomorphism. If $L_x\in Ker(\phi)$, then $[x,y]\in Leib(A)$ for all $y\in A$.  The last relation implies that $L_x\in Z(L(A))=Z(A/Leib(A))=\{0\}$.  Therefore, $\phi$ is injective and hence $L(A)$ is a subalgebra of $D(A)/I$. Moreover, for any $f\in D(A)$ and $L_x\in L(A)$, relation \eqref{Lx} implies that $[f+I,L_x+I]=L_{f(x)}+I\in \phi(L(A))$, hence $L(A)$ is isomorphic to an ideal of $D(A)/I$.  Henceforth, identify $L(A)$ with its isomorphic image $\phi(L(A))$. To prove the equality $C_{D(A)/I}(L(A))=\{0\}$, let $f+I\in C_{D(A)/I}(L(A))$.  Then for all $x\in A$, we have $0=[f+I,L_x+I]=L_{f(x)}+I$.  Therefore, $L_{f(x)}\in Ker(\phi)=\{0\}$.  We conclude that $Im(f)\subseteq Leib(A)$, i.e., $f+I=0$. \end{proof}

In the sequel, we set $\fr{g}:=D(A)/I$. Since $\fr{g}$ has trivial center (Lemma \ref{lemma11}), all Lie algebras $D^n(\fr{g})$ have trivial center and the corresponding adjoint representations $\op{ad}:D^{n-1}(\fr{g})\rightarrow D^{n}(\fr{g})$ are faithfull (see for example \cite{Sh}). Then  we can form the subideal sequence

\begin{equation}\label{subideas} L(A)\unlhd \fr{g}\unlhd D^1(\fr{g})\unlhd D^2(\fr{g})\unlhd \cdots \unlhd D^n(\fr{g})\unlhd \cdots,\end{equation}

\noindent where $L(A)$ is embedded as an ideal of $\fr{g}=D(A)/I$ via the map $L_x\mapsto L_x+I$, and each $D^{n-1}(\fr{g})$ is embedded as an ideal of $D^n(\fr{g})$ via the adjoint representation, $n\geq 1$.   We proceed to prove Theorem \ref{complete}.\\

\noindent \emph{Proof of Theorem \ref{complete}.} Recall that $A/Leib(A)=L(A)$.  Firstly, we treat the case $n=0$, i.e., show that $D(A)/I$ can be embedded as a subalgebra of $D(L(A))$.  To this end, define a map $\phi:D(A)\rightarrow D(L(A))$ with $\phi(f)(L_x):=L_{f(x)}$, $L_x\in L(A)$.  The map $\phi$ is well - defined, since for all $L_x,L_y\in L(A)$ and $f\in D(A)$, we have $\phi(f)([L_x,L_y])=\phi(f)(L_{[x,y]})=L_{[f(x),y]}+L_{[x,f(y)]}=[\phi(f)(L_x),L_y]+[L_x,\phi(f)(L_y)]$, i.e., $\phi(f)$ is a derivation of $L(A)$.  Moreover, $\phi$ is a homomorphism of Lie algebras since for all $f,g\in D(A)$ and $L_x\in L(A)$, we have

\begin{eqnarray*} [\phi(f),\phi(g)](L_x)&=&\phi(f)(L_{g(x)})-\phi(g)(L_{f(x)})=L_{f(g(x))}-L_{g(f(x))}\\
&=&L_{[f,g](x)}=\phi([f,g])(L_x).\end{eqnarray*}

\noindent Finally, if $f\in Ker(\phi)$ then $L_{f(x)}=0$, and hence $f(x)\in Leib(A)$.  Therefore, $f\in I$ and thus $Ker(\phi)=I$.  This concludes the proof for the case $n=0$.

 For $n>0$, since $A$ is perfect, it is clear that the Lie algebra $L(A)=A/Leib(A)$ is also perfect as a quotient of $A$.  In view of the subinvariant series \eqref{subideas}, Theorem \ref{th1} (and Corollary \ref{corol1}) imply that the Lie algebra $L(A)$ is an ideal of $D^{n-1}(\fr{g})$ for all $n\geq 1$.  In fact, Theorem \ref{th1} also implies that $L(A)$ is a characteristic ideal of $D^{n-1}(\fr{g})$ for all $n\geq 1$.  Therefore, in view of Proposition \ref{lemma10}, to show that $D^n(\fr{g})$ (i.e., $D(D^{n-1}(\fr{g}))$) is a subalgebra of $D(L(A))$ for all $n\geq 1$, it suffices to show that 

\begin{equation}\label{centralizer}C_{D^{n-1}(\fr{g})}(L(A))=\{0\},  \ \  \makebox{for all} \ \ n\geq 1.\end{equation}

\noindent For $n=1$, the fact that $C_{\fr{g}}(L(A))=0$ is shown in Lemma \ref{lemma12}. Moreover, for all $n\geq 1$, we have $C_{D^{n}(\fr{g})}(D^{n-1}(\fr{g}))=0$ (see for example Lemma 10 in \cite{Sh}).  Finally, if $A_1\unlhd A_2\unlhd \cdots \unlhd A_n$ is a subinvariant series of Lie algebras such that the centralizer of $A_{i}$ in $A_{i+1}$ is trivial for all $i=1,\dots,n-1$, then the centralizer of $A_1$ in $A_n$ is trivial (see for example Lemma 12 in \cite{Sh}). Applying the above to the subinvariant series \eqref{subideas}, we conclude that relation \eqref{centralizer} is true, thus completing the proof of Theorem \ref{complete}.\qed  \\

 A corollary of Theorem \ref{complete} is the following theorem of Su and Zhu.
 
 \begin{theorem}\label{SuZhu} \emph{(\cite{SuZhu})} Let $A$ be a perfect Lie algebra with trivial center.  Then the Lie algebra of derivations of $A$ is complete.   \end{theorem}

 Indeed, if $A$ is a perfect Lie algebra with trivial center then $Z^l(A)=Leib(A)=\{0\}$, and thus $A=L(A)=A/Leib(A)$ and $I=\{0\}$.  Theorem \ref{complete} yields $D(D(A))\subseteq D(A)$. On the other hand, since $A$ has trivial center, Lemma \ref{lemma11} implies that $D(A)$ also has trivial center, and thus $D(A)\unlhd D(D(A))$.  Therefore, $D(D(A))=D(A)$, implying that $D(A)$ is complete. \\

Now set $I^2:=\{F\in D^2(A):Im(F)\subseteq I\}$, where $I$ is the characteristic ideal of $D(A)$ given in Proposition \ref{ideal}. Let us state one more corollary of Theorem \ref{complete}.

\begin{corollary}\label{d^2}If $A$ is a perfect Leibniz algebra and the center of $A/Leib(A)$ is trivial, then the Lie algebra $D^2(A)/I^2$ can be embedded as a subalgebra of $D(A/Leib(A))$. \end{corollary}

\begin{proof} We define a map $\phi:D^2(A)\rightarrow D(D(A)/I)$ with $\phi(F)(f+I):=F(f)+I$, $F\in D^2(A)$, $f\in D(A)$.  Since $I$ is a characteristic ideal of $D(A)$ (Proposition \ref{ideal}), we have $\phi(F)(f+I)=\phi(F)(g+I)$ if $f-g\in I$, $F\in D^2(A)$. Besides, it is not hard to see that $\phi(F)$ defines a derivation of $D(A)/I$, and hence $\phi$ is well - defined. Moreover, it can be easily verified that $\phi$ is a homomorphism and $Ker(\phi)=I^2$.  Therefore, $D^2(A)/I^2$ can be embedded as a subalgebra of $D(D(A)/I)$.  By Theorem \ref{complete}, it is also a subalgebra of $D(A/Leib(A))$. \end{proof}

\begin{remark} One can observe that Theorem \ref{SuZhu} follows also as a corollary of the above result (if $Leib(A)=\{0\}$ then $I^2=0$ and thus $D^2(A)\subseteq D(A)$, i.e., $D(A)$ is complete).  Theorem \ref{SuZhu} essentially states that in the case of perfect Lie algebras $A$ with trivial center, the derivation tower $A\unlhd D(A)\unlhd D^2(A)\unlhd \cdots $ terminates at the second step (i.e., $D^2(A)=D(A)$ and thus $D(A)$ is complete), irrespectively of the dimension of the algebra and the characteristic of the ground field. Although Theorem \ref{complete} and Corollary \ref{d^2} do not immediately generalize this property to perfect Leibniz algebras, they do imply that the derivation algebras $D^n(D(A)/I)$ and $D^2(A)/I^2$ are subalgebras of $D(A/Leib(A))$ (which is a complete Lie algebra by Theorem \ref{SuZhu}). Within this context, it would be interesting to investigate whether there exists an (in a sense uniform) $n_0\in \mathbb N$ such that $D^{n_0}(D(A)/I)$ is a complete Lie algebra for all perfect Leibniz algebras $A$ with $Z(A/Leib(A))=\{0\}$, irrespectively of the dimension of $A$ and the characteristic of the ground field.   For perfect Lie algebras, Theorem \ref{SuZhu} ensures that $n_0=0$.  \end{remark}

\end{document}